\documentclass{article}

\usepackage{amsthm}
\usepackage{amsmath, amssymb}
\usepackage{hyperref}
\usepackage{tikz}
\usetikzlibrary{matrix}
\usepackage[matrix, arrow, curve]{xy}

\newtheorem{Thm}{Theorem}[section]
\newtheorem{Def}[Thm]{Definition}

\newtheorem{Kor}[Thm]{Corollary}

\newtheorem{Assumption}[Thm]{General Assumptions}

\title{Vector bundles over certain Koras-Russell threefolds of the third kind}

\author{Tariq Syed\\
Mathematisches Institut\\
Heinrich-Heine-Universit{\"a}t D{\"u}sseldorf\\
Universit{\"a}tsstra{\ss}e 1\\
40225 D{\"u}sseldorf, Germany\\
tariq.syed@gmx.de}

\date{\today} 

\begin{document}

\maketitle

\begin{abstract}
Let $k$ be an algebraically closed base field of characteristic $0$ and let $\alpha_{1}, \alpha_{2}, \alpha_{3}, d \geq 2$ be integers such that $\alpha_{1}, \alpha_{2}, \alpha_{3}$ are pairwise coprime and $gcd (\alpha_{1},d-1) = 1$. Then consider the Koras-Russell threefold $Y := \{ x + x^d y^{\alpha_{1}} + z^{\alpha_{2}} + t^{\alpha_{3}} = 0\} \subset \mathbb{A}^{4}_{k}$. We prove that the Chow groups $CH^{i}(Y)$ are trivial for $i=1,2,3$ and therefore all algebraic vector bundles over $Y$ are trivial. If $\alpha_{1}$ is odd, we also prove that the Chow-Witt groups $\widetilde{CH}^{i}(Y, \mathcal{L})$ are trivial for $i=1,2,3$ and any line bundle $\mathcal{L}$ over $Y$.
\end{abstract}

\tableofcontents

\section{Introduction}

Koras-Russell threefolds were first discovered in the context of the linearization problem for $\mathbb{G}_{m}$-actions on $\mathbb{A}^3_{\mathbb{C}}$ (cf. \cite[Theorem 4.1]{KR}). These threefolds are topologically contractible smooth affine $\mathbb{C}$-varieties of dimension $3$ which are not isomorphic to $\mathbb{A}^3_{\mathbb{C}}$ and have since been studied in the several other contexts such as the Zariski cancellation problem and motivic homotopy theory. They can roughly be divided into three types via their Makar-Limanov invariants (cf. \cite{KML}) and, in particular, via $\mathbb{G}_{a}$-actions on them:\\
First of all, any variety which is isomorphic to a subvariety of $\mathbb{A}^4_{\mathbb{C}}$ of the form
\begin{center}
$\{x + x^d y + z^{\alpha_2} + t^{\alpha_3}=0\} \subset \mathbb{A}^4_{\mathbb{C}}= Spec(\mathbb{C}[x,y,z,t])$,
\end{center}
where $d, \alpha_{2}, \alpha_{3} \geq 2$ and $gcd(\alpha_{2},\alpha_{3})=1$, is called a Koras-Russell threefold of the first kind. Secondly, any variety which is isomorphic to a subvariety of $\mathbb{A}^4_{\mathbb{C}}$ of the form
\begin{center}
$\{x + {(x^d + z^{\alpha_{2}})}^{l} y + t^{\alpha_3}=0\} \subset \mathbb{A}^4_{\mathbb{C}}= Spec(\mathbb{C}[x,y,z,t])$,
\end{center}
where $l \geq 1$, $d, \alpha_{2}, \alpha_{3} \geq 2$ and $gcd(\alpha_{2},d\alpha_{3})=1$, is called a Koras-Russell threefold of the second kind.\\
Following the procedure outlined in \cite[Theorem 4.1]{KR}, one may also construct Koras-Russell threefolds which are not isomorphic to any Koras-Russell threefolds of the first or second kind; we will refer to such varieties as Koras-Russell threefolds of the third kind. In contrast to Koras-Russell threefolds of the first or second kind, Koras-Russell threefolds of the third kind do not admit non-trivial $\mathbb{G}_{a}$-actions.\\
In \cite[Question 7.12]{KR}, M. Koras and P. Russell posed the following question:\\\\
\textbf{Question 1.} If $Y$ is a Koras-Russell threefold, are all algebraic vector bundles over $Y$ trivial?\\

More generally, one may ask whether algebraic vector bundles over topologically contractible smooth affine $\mathbb{C}$-varieties are always trivial or not; this problem is also called the generalized Serre question in the literature (cf. \cite[Question 6]{AO}). The generalized Serre question has a positive answer in dimensions $\leq 2$, but is completely open in higher dimensions (cf. \cite[Section 5.5.2]{AO}). A first major result on vector bundles over Koras-Russell threefolds was proven by M. P. Murthy, who proved that algebraic vector bundles over Koras-Russell threefolds of the first kind are always trivial (cf. \cite[Corollary 3.8]{M}). Later Koras-Russell threefolds were revisited in the context of motivic homotopy theory. Generalizing M. P. Murthy's results, M. Hoyois, A. Krishna and P. A. {\O}stv{\ae}r proved that algebraic vector bundles over Koras-Russell threefolds of the first or second kind are always trivial (cf. \cite[Corollary 3.7]{HKO}). No analogous results were proven on Koras-Russell threefolds of the third kind.\\
In this paper, we consider the varieties

\begin{center}
$Y (d, \alpha_{1}, \alpha_{2}, \alpha_{3}) := \{ x + x^d y^{\alpha_{1}} + z^{\alpha_{2}} + t^{\alpha_{3}} = 0\} \subset \mathbb{A}^{4}_{k} = Spec (k[x,y,z,t])$,
\end{center}

where $k$ is an algebraically closed field of characteristic $0$ and $d, \alpha_{1}, \alpha_{2}, \alpha_{3} \geq 2$ are integers such that $\alpha_{1}, \alpha_{2}, \alpha_{3}$ are pairwise coprime and $gcd (\alpha_{1},d-1) = 1$. If $k = \mathbb{C}$, the variety $Y (d, \alpha_{1}, \alpha_{2}, \alpha_{3})$ is a Koras-Russell threefold and already appeared in \cite[Example 7.6(2)]{KR}; if furthermore $\alpha_{1}, \alpha_{2}, \alpha_{3} \gg 0$, then $Y (d, \alpha_{1}, \alpha_{2}, \alpha_{3})$ has non-negative logarithmic Kodaira dimension by \cite[Proposition 6.5]{KR} and therefore has to be a Koras-Russell threefold of the third kind (Koras-Russell threefolds of the first and second kind have logarithmic Kodaira dimension $- \infty$). We study the Chow groups of $Y (d, \alpha_{1}, \alpha_{2}, \alpha_{3})$ and prove the following result (cf. Theorems \ref{Chow-3}, \ref{Chow-1-2} and Corollary \ref{VB} in the text):\\\\
\textbf{Theorem 2.} One has $CH^{i}(Y (d, \alpha_{1}, \alpha_{2}, \alpha_{3}))=0$ for $i=1,2,3$. In particular, all algebraic vector bundles over $Y (d, \alpha_{1}, \alpha_{2}, \alpha_{3})$ are trivial.\\

The statement on vector bundles in Theorem 2 follows from the vanishing statement on Chow groups together with well-known classification results on algebraic vector bundles over smooth affine threefolds over algebraically closed fields (cf. \cite[Theorem 2.1(iii)]{KM}, \cite[Theorem 1]{AF}) and answers Question 1 in case $k=\mathbb{C}$ and $Y = Y (d, \alpha_{1}, \alpha_{2}, \alpha_{3})$.\\
Note that the computation of Chow groups or more general cohomology groups of topologically contractible smooth affine $\mathbb{C}$-varieties is a difficult problem in its own right. M. Hoyois, A. Krishna and P. A. {\O}stv{\ae}r proved in \cite[Theorem 4.2]{HKO} that Koras-Russell threefolds of the first or second kind are stably $\mathbb{A}^1$-contractible; A. Dubouloz and J. Fasel even proved that Koras-Russell threefolds of the first kind are $\mathbb{A}^1$-contractible (cf. \cite[Theorem 1]{DF}). This means in particular that any Koras-Russell threefold of the first or second kind has the cohomology of $Spec(\mathbb{C})$ for any cohomology theory representable in the stable motivic homotopy category $\mathcal{SH}(\mathbb{C})$ of $\mathbb{P}^1$-spectra (e.g., Chow groups or Chow-Witt groups). It is completely open whether Koras-Russell threefolds of the third kind are stably $\mathbb{A}^1$-contractible as well. A necessary condition for a Koras-Russell threefold $Y$ of the third kind to be stably $\mathbb{A}^1$-contractible is that $CH^{i}(Y) = 0$ for $i=1,2,3$. Theorem 2 thus verifies this necessary condition for $Y = Y (d, \alpha_{1}, \alpha_{2}, \alpha_{3})$. Analogously, another necessary condition for a Koras-Russell threefold $Y$ of the third kind to be stably $\mathbb{A}^1$-contractible is that its Chow-Witt groups $\widetilde{CH}^{i}(Y, \mathcal{L})$ vanish for $i=1,2,3$ and any line bundle $\mathcal{L}$ over $Y$. We thus also study the Chow-Witt groups $\widetilde{CH}^{i}(Y (d, \alpha_{1}, \alpha_{2}, \alpha_{3}), \mathcal{L})$ of $Y (d, \alpha_{1}, \alpha_{2}, \alpha_{3})$ and prove the following statement (cf. Theorem \ref{Chow-Witt} in the text):\\\\
\textbf{Theorem 3.} Assume $\alpha_{1}$ is odd. One has $\widetilde{CH}^{i}(Y (d, \alpha_{1}, \alpha_{2}, \alpha_{3}), \mathcal{L})=0$ for $i=1,2,3$ and any line bundle $\mathcal{L}$ over $Y (d, \alpha_{1}, \alpha_{2}, \alpha_{3})$.

\section{Preliminaries}\label{Preliminaries}\label{2}

In this section, we briefly review the results on the motivic cohomology of cyclic coverings obtained in \cite{Sy} as needed for this paper. We first explain the general setting of those results (cf. \cite[Definition 1.1 and General Assumptions 1.2]{Sy}):

\begin{Def}\label{CyclicCoverings}
Let $k$ be a field of characteristic $0$, let $X$ be an affine $k$-variety and let $f \in \mathcal{O}_{X}(X) \setminus \{0\}$ be a regular function and let $s > 1$ be an integer. Then
\begin{itemize}
\item we denote by $F_0$ the closed subscheme of $X$ defined by $f$;
\item we let $Y_{s}$ be the closed subscheme of $X \times_k \mathbb{A}^1_k$ defined by $f = u^s$, where $u$ is the variable of $\mathbb{A}^1_k$;
\item we will denote by $F$ the closed subscheme of $Y_s$ defined by $u=0$.
\end{itemize}
We call the projection $\varphi_s : Y_s \rightarrow X$ a cyclic covering of $X$ branched to order $s$ along $F_0$.
\end{Def}

We will use the notation from Definition \ref{CyclicCoverings} throughout this section. The following general assumptions are made in \cite{Sy}:

\begin{Assumption}\label{Assumptions}
We furthermore assume that
\begin{itemize}
\item $k$ is algebraically closed;
\item $u^s - f$ is prime in the polynomial rings $\mathcal{O}_{X}(X)[u]$ and $k(X)[u]$, where $k(X)$ is the field of fractions of the domain $\mathcal{O}_{X}(X)$;
\item $X$ and $F_0$ (and hence $Y_s$ by the remarks in \cite[Definition 5.1]{Z}) are smooth over $k$.
\end{itemize}
\end{Assumption}

Note that the cyclic covering $\varphi_{s}$ maps the subscheme $F \subset Y_{s}$ isomorphically onto the subscheme $F_{0} \subset X$ (cf. \cite[Section 3.1]{Sy}). With the notation from Definition \ref{CyclicCoverings} and under the General Assumptions \ref{Assumptions} above, we now list the results from \cite{Sy} needed in the proof of the main results of this paper in the next section. Let $Sm_k$ denote the category of smooth $k$-schemes. In the following statements, we will denote the motivic cohomology group of a smooth $k$-scheme $V \in Sm_k$ in bidegree $(i,j) \in \mathbb{Z}^2$ with coefficients in a fixed commutative ring $R$ by $H^{i,j}(V, R)$.

\begin{Thm}[{{\cite[Theorem 3.17]{Sy}}}]\label{Ind-Thm}
Let $(p,q) \in \mathbb{Z}^2$. Assume that
\begin{itemize}
\item[(a)] $s \in R^{\times}$
\item[(b)] there is a $\mathbb{G}_{m,k}$-action $\gamma$ on $X$, which makes $f$ a quasi-invariant of weight $d \in \mathbb{Z}$ with respect to $\gamma$ and $\langle d,s \rangle = \mathbb{Z}$
\item[(c)] $H^{p-1,q-1}(Y_{s} \setminus F,R)=0$
\end{itemize}
Then the morphism $\varphi_{s}: Y_{s} \rightarrow X$ induces an isomorphism $H^{p,q}(X,R)\cong H^{p,q}(Y_{s},R)$.
\end{Thm}

The next theorem is a consequence of Theorem \ref{Ind-Thm} by means of the identification $CH^{i}(V) \otimes_{\mathbb{Z}} R \cong H^{2i,i}(V,R)$ for $V \in Sm_k$ and \cite[Vanishing Theorem 19.3]{MVW}.

\begin{Thm}[{{\cite[Theorem 3.19]{Sy}}}]\label{Chow-Tensor}
Assume that
\begin{itemize}
\item[(a)] $s \in R^{\times}$
\item[(b)] there is a $\mathbb{G}_{m,k}$-action $\gamma$ on $X$, which makes $f$ a quasi-invariant of weight $d \in \mathbb{Z}$ with respect to $\gamma$ and $\langle d,s \rangle = \mathbb{Z}$
\end{itemize}
Then $\varphi_{s}$ induces isomorphisms $CH^{i}(X)\otimes_{\mathbb{Z}}R \cong CH^{i}(Y_{s})\otimes_{\mathbb{Z}}R$ for $i \geq 0$.
\end{Thm}

Finally, the next result from \cite{Sy} is also a consequence of Theorem \ref{Ind-Thm} and \cite[Vanishing Theorem 19.3]{MVW}:

\begin{Kor}[{{\cite[Corollary 3.24]{Sy}}}]\label{H53}
Assume that
\begin{itemize}
\item[(a)] $s \in R^{\times}$
\item[(b)] there is a $\mathbb{G}_{m,k}$-action $\gamma$ on $X$, which makes $f$ a quasi-invariant of weight $d \in \mathbb{Z}$ with respect to $\gamma$ and $\langle d,s \rangle = \mathbb{Z}$
\item[(c)] $H^{2i,i} (X,R) = 0$ for $i \geq 1$
\end{itemize}
Then the morphism $\varphi_{s}: Y_{s} \rightarrow X$ induces an isomorphism $H^{2i+1,i+1}(X,R) \cong H^{2i+1,i+1}(Y_{s},R)$ for $i \geq 1$.
\end{Kor}

For any integer $j \geq 0$ and $V \in Sm_k$, let $\textbf{I}^j$ denote the sheaf on the small Nisnevich site of $V$ associated to the $j$th power of the fundamental ideal in the Witt ring and let $\overline{\textbf{I}}^{j}= \textbf{I}^{j}/\textbf{I}^{j+1}$; for more background on these sheaves, we refer the reader to \cite[\S 2.1]{F}. We conclude this section with a long exact sequence from \cite{T} involving motivic cohomology groups with $\mathbb{Z}/2\mathbb{Z}$-coefficients and sheaf cohomology groups with coefficients in the sheaves $\overline{\textbf{I}}^j$.

\begin{Thm}[{{\cite[Theorem 1.3]{T}}}]\label{Totaro}
For any smooth $k$-scheme $V$ and any integers $i,j \geq 0$, there is a long exact sequence
\begin{center}
$... \rightarrow H^{i+j,j-1} (V,\mathbb{Z}/2\mathbb{Z}) \rightarrow H^{i+j,j} (V,\mathbb{Z}/2\mathbb{Z}) \rightarrow H^{i}_{Nis}(V,\overline{\textbf{I}}^j) \rightarrow H^{i+j+1,j-1} (V,\mathbb{Z}/2\mathbb{Z}) \rightarrow ...$.
\end{center}
\end{Thm}

\section{Results}\label{Results}\label{3}

In this section we prove the main results of this paper. We assume familiarity with Chow groups and Chow-Witt groups for brevity. We let $k$ be an algebraically closed base field of characteristic $0$ and we fix integers $\alpha_{1}, \alpha_{2}, \alpha_{3}, d \geq 2$ with $\alpha_{1}, \alpha_{2}, \alpha_{3}$ pairwise coprime and $gcd (\alpha_{1},d-1) = 1$. Then we consider the variety

\begin{center}
$Y := \{ x + x^d y^{\alpha_{1}} + z^{\alpha_{2}} + t^{\alpha_{3}} = 0\} \subset \mathbb{A}^{4}_{k}$.
\end{center}

If $k = \mathbb{C}$, then $Y$ is a Koras-Russell threefold by \cite[Example 7.6(2)]{KR}.

\begin{Thm}\label{Chow-3}
$CH^{i}(Y) \otimes_{\mathbb{Z}}\mathbb{Q}=0$ for $i =1,2,3$ and $CH^{3}(Y) = 0$.
\end{Thm}

\begin{proof}
First of all, the variety $Y$ can be interpreted as a cyclic covering of the Koras-Russell threefold of the first kind
\begin{center}
$X := \{x + x^d y + z^{\alpha_2} + t^{\alpha_3}=0\} \subset \mathbb{A}^4_{k}$
\end{center}
with respect to the regular function $y \in \mathcal{O}_X (X)$ of order $\alpha_1$. Furthermore, the $\mathbb{G}_{m,k}$-action on $X$ given by
\begin{center}
$(\lambda,x,y,z,t) \mapsto (\lambda^{\alpha_{2}\alpha_{3}}x,\lambda^{-(d-1)\alpha_{2}\alpha_{3}}y,\lambda^{\alpha_{3}}z,\lambda^{\alpha_{2}}t)$
\end{center}
makes $y$ a quasi-invariant of weight $-(d-1)\alpha_{2}\alpha_{3}$. Since $\alpha_{1} \in \mathbb{Q}^{\times}$, Theorem \ref{Chow-Tensor} implies that $CH^i (Y)\otimes_{\mathbb{Z}}\mathbb{Q} \cong CH^i (X)\otimes_{\mathbb{Z}}\mathbb{Q}$ for $i \geq 1$. The main result in \cite{DF} shows that $X$ is $\mathbb{A}^1$-contractible; in particular, $CH^{i}(X) = 0$ for $i=1,2,3$ and therefore $CH^{i}(Y)\otimes_{\mathbb{Z}}\mathbb{Q} = 0$ for $i=1,2,3$.\\
Now recall from \cite{Sr} that $CH^3 (Y)$ is a uniquely divisible abelian group and therefore a $\mathbb{Q}$-vector space. In particular, there is a $\mathbb{Q}$-linear isomorphism of the form $CH^3 (X) \cong \mathbb{Q}^I$ for some index set $I$. The previous paragraph shows that $CH^3 (X) \otimes_{\mathbb{Z}} \mathbb{Q} = 0$. Now it follows from the canonical isomorphism $\mathbb{Q} \otimes_{\mathbb{Z}} \mathbb{Q} \cong \mathbb{Q}$ that $0 = CH^3 (Y) \otimes_{\mathbb{Z}} \mathbb{Q} \cong \mathbb{Q}^I \otimes_{\mathbb{Z}} \mathbb{Q} \cong \mathbb{Q}^I \cong CH^3 (Y)$. This finishes the proof.
\end{proof}

\begin{Thm}\label{Chow-1-2}
$CH^{i}(Y) = 0$ for $i=1,2$.
\end{Thm}

\begin{proof}
First of all, the variety $Y$ can be interpreted as a cyclic covering of $\mathbb{A}^3_{k}$ (with coordinates $x,y,z$) with respect to the regular function
\begin{center}
$f_{1} = -(x + x^d y^{\alpha_{1}} + z^{\alpha_{2}})$
\end{center}
of order $\alpha_{3}$. The $\mathbb{G}_{m,k}$-action on $\mathbb{A}^3_{k}$ given by
\begin{center}
$(\lambda,x,y,z) \mapsto (\lambda^{\alpha_{1}\alpha_{2}}x,\lambda^{-(d-1)\alpha_{2}}y,\lambda^{\alpha_{1}}z)$
\end{center}
makes $f_{1}$ a quasi-invariant of weight $\alpha_{1}\alpha_{2}$. Therefore Theorem \ref{Chow-Tensor} implies that $CH^{i}(Y)$, $i \geq 1$, are divisible by any integer $n$ with $\langle n,\alpha_{3}\rangle=\mathbb{Z}$ (just take $R = \mathbb{Z}/n\mathbb{Z}$).\\
Analogously, the variety $Y$ can also be interpreted as a cyclic covering of $\mathbb{A}^3_{k}$ (with coordinates $x,y,t$) with respect to the regular function
\begin{center}
$f_{2} = -(x + x^d y^{\alpha_{1}} + t^{\alpha_{3}})$
\end{center}
of order $\alpha_{2}$. The $\mathbb{G}_{m,k}$-action on $\mathbb{A}^3_{k}$ given by
\begin{center}
$(\lambda,x,y,t) \mapsto (\lambda^{\alpha_{1}\alpha_{3}}x,\lambda^{-(d-1)\alpha_{3}}y,\lambda^{\alpha_{1}}t)$
\end{center}
makes $f_{2}$ a quasi-invariant of weight $\alpha_{1}\alpha_{3}$. Again, Theorem \ref{Chow-Tensor} implies that $CH^{i}(Y)$, $i \geq 1$, are also divisible by any integer $n$ with $\langle n,\alpha_{2}\rangle=\mathbb{Z}$. Since $\alpha_{2}$ and $\alpha_{3}$ are coprime, the groups $CH^{i}(Y)$, $i \geq 1$, are divisible by any prime number and hence are divisible abelian groups. As a consequence, it suffices to show that there are integers $m,n \neq 0$ such that $CH^{1}(Y)$ is $m$-torsion and $CH^{2}(Y)$ is $n$-torsion.\\
For this purpose, we let
\begin{center}
$F_{1}:= \{t=0\} \subset Y$.
\end{center}
Note that the cyclic covering $Y \rightarrow \mathbb{A}^3_{k}$ from the first paragraph maps the variety $F_{1}$ isomorphically onto the subvariety
\begin{center}
$F_{2}:=\{-(x + x^d y^{\alpha_{1}} + z^{\alpha_{2}})=0\} \subset \mathbb{A}^3_{k}$.
\end{center}
It follows from \cite[Remark 3.21]{Sy} that $CH^{i}(Y \setminus F_{1})$ is an $\alpha_{3}$-torsion group for $i \geq 1$. Now consider the localization sequence for Chow groups
\begin{center}
$CH^{0}(F_{1}) \rightarrow CH^{1} (Y) \rightarrow CH^{1} (Y\setminus F_{1}) \rightarrow 0$.
\end{center}
The group $CH^0 (F_{1}) \cong CH^0 (F_{2}) \cong \mathbb{Z}$ cannot map injectively into $CH^{1}(Y)$ because tensoring the localization sequence with $\mathbb{Q}$ would then automatically yield $CH^{1}(Y) \otimes_{\mathbb{Z}} \mathbb{Q} \neq 0$, which would be a contradiction to Theorem \ref{Chow-3}. Therefore one obtains a short exact sequence of the form
\begin{center}
$0 \rightarrow \mathbb{Z}/l\mathbb{Z} \rightarrow CH^{1} (Y) \rightarrow CH^{1} (Y\setminus F_{1}) \rightarrow 0$
\end{center}
for some integer $l \neq 0$. This implies that $CH^1 (Y)$ is $m$-torsion for $m=l\alpha_{3}$. This proves that $CH^1 (Y) = 0$.\\
It remains to be proven that $CH^2 (Y) = 0$. Note that $F_{2}$ is a cyclic covering of $\mathbb{A}^2_{k}$ of order $\alpha_{2}$ with respect to the regular function $-(x + x^d y^{\alpha_{1}})$. If we then let
\begin{center}
$F_{3}:= \{z=0\} \subset F_{2}$,
\end{center}
\cite[Remark 3.21]{Sy} therefore implies that $CH^{1}(F_{2}\setminus F_{3})$ is $\alpha_{2}$-torsion. Then we consider the exact localization sequence for Chow groups
\begin{center}
$CH^{0}(F_{3}) \rightarrow CH^{1} (F_{2}) \rightarrow CH^{1} (F_{2}\setminus F_{3}) \rightarrow 0$.
\end{center}
Note that Theorem \ref{Chow-Tensor} applied to the cyclic covering $F_{2} \rightarrow \mathbb{A}^2_{k}$ of order $\alpha_{2}$ induced by projection to the coordinates $x$ and $y$ yields
\begin{center}
$CH^{1}(F_{2}) \otimes_{\mathbb{Z}} \mathbb{Q} = 0$.
\end{center}
Note that $CH^0 (F_{3})$ is a finite direct sum $\bigoplus_{i=1}^{r} \mathbb{Z}$ of copies of $\mathbb{Z}$ and neither of the induced maps $\mathbb{Z} \hookrightarrow \bigoplus_{i=1}^{r} \mathbb{Z} \rightarrow CH^{1}(F_{1})$ can be injective because tensoring with $\mathbb{Q}$ would then yield $CH^{1}(F_{1}) \otimes_{\mathbb{Z}} \mathbb{Q} \neq 0$. So we obtain an exact sequence of the form
\begin{center}
$\bigoplus_{i=1}^{r} \mathbb{Z}/l_{i}\mathbb{Z} \rightarrow CH^{1} (F_{2}) \rightarrow CH^{1} (F_{2}\setminus F_{3}) \rightarrow 0$
\end{center}
for some integers $l_{i} \neq 0$, $1 \leq i \leq r$. This implies that $CH^{1} (F_{2})$ is $l\alpha_{2}$-torsion for $l = \prod_{i}^{r} l_{i}$. In particular, $CH^{1} (F_{1})$ is $l\alpha_{2}$-torsion (as $F_{1} \cong F_{2}$).\\
Altogether, the previous paragraphs and the exact localization sequence
\begin{center}
$CH^{1}(F_{1}) \rightarrow CH^{2} (Y) \rightarrow CH^{2} (Y\setminus F_{1}) \rightarrow 0$
\end{center}
imply that $CH^{2}(Y)$ is $n$-torsion for $n=l\alpha_{2}\alpha_{3}$. This finishes the proof.
\end{proof}

\begin{Kor}\label{VB}
All algebraic vector bundles over $Y$ are trivial.
\end{Kor}

\begin{proof}
It follows from \cite[Theorem 2.1(iii)]{KM} and \cite[Theorem 1]{AF} that algebraic vector bundles over $Y$ are completely classified by their rank and their Chern classes in $CH^{1}(Y), CH^{2}(Y)$ and $CH^{3}(Y)$. We refer the reader to \cite[Section 2.4]{Sy} for details on these classification results. Therefore the vanishing of the Chow groups of $Y$ directly implies that all algebraic vector bundles over $Y$ are trivial.
\end{proof}

Theorems \ref{Chow-3}, \ref{Chow-1-2} and Corollary \ref{VB} establish Theorem 2. The next result establishes Theorem 3:

\begin{Thm}\label{Chow-Witt}
Assume $\alpha_{1}$ is odd and let $\mathcal{L}$ be a line bundle over $Y$. Then $\widetilde{CH}^{i}(Y, \mathcal{L}) = 0$ for $i=1,2,3$.
\end{Thm}

\begin{proof}
For $i=1,2,3$, we consider the short exact sequences
\begin{center}
$0 \rightarrow \textbf{I}^{i+1}(\mathcal{L}) \rightarrow \textbf{K}^{MW}_{i}(\mathcal{L}) \rightarrow \textbf{K}^{M}_{i} \rightarrow 0$
\end{center}
of sheaves on the small Nisnevich site on $Y$ (e.g., cf. \cite[\S 2.1]{F}). We have identifications $\widetilde{CH}^{i}(Y, \mathcal{L}) \cong H^{i}_{Nis}(Y, \textbf{K}^{MW}_{i}(\mathcal{L}))$. Theorems \ref{Chow-3} and \ref{Chow-1-2} and the fact that $H^{i}(Y, \textbf{K}^{M}_{i})\cong CH^{i}(Y)$ then imply via the long exact sequence of sheaf cohomology groups that it suffices to show that
\begin{itemize}
\item[1)] $H^{3}(Y, \textbf{I}^{4}(\mathcal{L}))=0$,
\item[2)] $H^{2}(Y, \textbf{I}^{3}(\mathcal{L}))=0$,
\item[3)] $H^{1}(Y, \textbf{I}^{2}(\mathcal{L}))=0$.
\end{itemize}
The first vanishing statement follows from \cite[Proposition 5.1]{AF} and the second one from \cite[Proposition 5.2]{AF} as $Y$ has dimension $3$. So it remains to be proven that $H^{1}(Y, \textbf{I}^{2}(\mathcal{L}))=0$.\\
For this purpose, we consider the short exact sequences
\begin{center}
$0 \rightarrow \textbf{I}^{3}(\mathcal{L}) \rightarrow \textbf{I}^{2}(\mathcal{L}) \rightarrow \overline{\textbf{I}}^{2} \rightarrow 0$
\end{center}
of sheaves on the small Nisnevich site on $Y$ (e.g., cf. \cite[\S 2.1]{F}). Its associated long exact sequence of sheaf cohomology groups shows that it suffices to prove that
\begin{itemize}
\item[1)] $H^{1}(Y, \overline{\textbf{I}}^{2})=0$,
\item[2)] $H^{1}(Y, \textbf{I}^{3}(\mathcal{L}))=0$.
\end{itemize}
Let us first prove the first statement: If $i=1$ and $j=2$, Theorem \ref{Totaro} gives a long exact sequence of the form
\begin{center}
$... \rightarrow H^{3,1} (Y,\mathbb{Z}/2\mathbb{Z}) \rightarrow H^{3,2} (Y,\mathbb{Z}/2\mathbb{Z}) \rightarrow H^{1}_{Nis}(Y,\overline{\textbf{I}}^2) \rightarrow H^{4,1} (Y,\mathbb{Z}/2\mathbb{Z}) \rightarrow ...$.
\end{center}
One has $H^{4,1} (Y,\mathbb{Z}/2\mathbb{Z})=0$ by \cite[Vanishing Theorem 19.3]{MVW}. If we interpret $Y$ as a cyclic covering of the Koras-Russell threefold of the first kind
\begin{center}
$X := \{x + x^d y + z^{\alpha_2} + t^{\alpha_3}=0\} \subset \mathbb{A}^4_{k}$
\end{center}
of order $\alpha_1$ with respect to the regular function $y \in \mathcal{O}_X (X)$, then Corollary \ref{H53} implies that $H^{3,2} (Y,\mathbb{Z}/2\mathbb{Z})=0$, because $X$ is $\mathbb{A}^1$-contractible by \cite[Theorem 1]{DF}. Therefore $H^{1}(Y, \overline{\textbf{I}}^{2})=0$, as desired.\\
So we have reduced the proof of Theorem \ref{Chow-Witt} to verifying the second vanishing statement above, i.e., $H^{1}(Y, \textbf{I}^{3}(\mathcal{L}))=0$. Note that the short exact sequence
\begin{center}
$0 \rightarrow \textbf{I}^{4}(\mathcal{L}) \rightarrow \textbf{I}^{3}(\mathcal{L}) \rightarrow \overline{\textbf{I}}^{3} \rightarrow 0$
\end{center}
of sheaves on the small Nisnevich site of $Y$  (cf. \cite[\S 2.1]{F}) and \cite[Proposition 5.1]{AF} show that $H^{1}(Y, \textbf{I}^{3}(\mathcal{L}))\cong H^{1}(Y, \overline{\textbf{I}}^{3})$. If $i=1$ and $j=3$, Theorem \ref{Totaro} gives a long exact sequence of the form
\begin{center}
$... \rightarrow H^{4,2} (Y,\mathbb{Z}/2\mathbb{Z}) \rightarrow H^{4,3} (Y,\mathbb{Z}/2\mathbb{Z}) \rightarrow H^{1}_{Nis}(Y,\overline{\textbf{I}}^3) \rightarrow H^{5,2} (Y,\mathbb{Z}/2\mathbb{Z}) \rightarrow ...$.
\end{center}
One has $H^{5,2} (Y,\mathbb{Z}/2\mathbb{Z})=0$ by \cite[Vanishing Theorem 19.3]{MVW} and $H^{4,2} (Y,\mathbb{Z}/2\mathbb{Z}) \cong CH^{2}(Y)\otimes_{\mathbb{Z}}\mathbb{Z}/2\mathbb{Z} = 0$ by Theorem \ref{Chow-1-2}. In particular,
\begin{center}
$H^{1}(Y, \textbf{I}^{3}(\mathcal{L}))\cong H^{1}(Y, \overline{\textbf{I}}^{3})\cong H^{4,3} (Y,\mathbb{Z}/2\mathbb{Z})$
\end{center}
and it suffices to show that $H^{4,3} (Y,\mathbb{Z}/2\mathbb{Z})=0$.\\
For this purpose, we interpret $Y$ again as a cyclic covering of the Koras-Russell threefold of the first kind
\begin{center}
$X := \{x + x^d y + z^{\alpha_2} + t^{\alpha_3}=0\} \subset \mathbb{A}^4_{k}$
\end{center}
of order $\alpha_1$ with respect to the regular function $y \in \mathcal{O}_X (X)$. Then the closed subscheme
\begin{center}
$F := \{y=0\} \subset Y$
\end{center}
is isomorphic to $\mathbb{A}^2_{k}$. We first note that $H^{2,1} (Y,\mathbb{Z}/2\mathbb{Z})\cong CH^{1}(Y)\otimes_{\mathbb{Z}}\mathbb{Z}/2\mathbb{Z} = 0$ by Theorem \ref{Chow-1-2}; hence also $H^{2,1} (Y \setminus F,\mathbb{Z}/2\mathbb{Z})\cong CH^{1}(Y \setminus F)\otimes_{\mathbb{Z}}\mathbb{Z}/2\mathbb{Z} = 0$. As $\alpha_{1}$ is odd and $X$ is $\mathbb{A}^1$-contractible, we conclude by Theorem \ref{Ind-Thm} that $H^{3,2}(Y, \mathbb{Z}/2\mathbb{Z})=0$. Since $F$ is isomorphic to $\mathbb{A}^2_{k}$, we have $H^{2,1}(F, \mathbb{Z}/2\mathbb{Z}) \cong H^{2,1}(\mathbb{A}^2_{k}, \mathbb{Z}/2\mathbb{Z}) = 0$ and therefore the long exact localization sequence of motivic cohomology groups
\begin{center}
$...\rightarrow H^{1,1}(F,\mathbb{Z}/2\mathbb{Z}) \rightarrow H^{3,2}(Y,\mathbb{Z}/2\mathbb{Z}) \rightarrow H^{3,2}(Y \setminus F, \mathbb{Z}/2\mathbb{Z}) \rightarrow H^{2,1}(F,\mathbb{Z}/2\mathbb{Z}) \rightarrow ...$
\end{center}
implies that $H^{3,2}(Y \setminus F, \mathbb{Z}/2\mathbb{Z})=0$. Therefore we can apply Theorem \ref{Ind-Thm} again and deduce that $H^{4,3}(Y, \mathbb{Z}/2\mathbb{Z})=0$. This shows that
\begin{center}
$H^{1}(Y, \textbf{I}^{3}(\mathcal{L}))\cong H^{1}(Y, \overline{\textbf{I}}^{3})\cong H^{4,3} (Y,\mathbb{Z}/2\mathbb{Z}) = 0$
\end{center}
and finishes the proof of Theorem \ref{Chow-Witt}.
\end{proof}


\begin{thebibliography}{xxxxxx}
\bibitem[AF]{AF} A. Asok, J. Fasel, A cohomological classification of vector bundles on smooth affine threefolds, Duke Math. Journal 163 (2014), no. 14, 2561-2601
\bibitem[A\O]{AO} A. Asok, P. A. {\O}stv{\ae}r, $\mathbb{A}^{1}$-homotopy Theory and Contractible Varieties: A Survey, In: Neumann, F., P{\'a}l, A. (eds) Homotopy Theory and Arithmetic Geometry – Motivic and Diophantine Aspects. Lecture Notes in Mathematics, vol 2292. Springer, Cham. \url{https://doi.org/10.1007/978-3-030-78977-0_5}
\bibitem[DF]{DF} A. Dubouloz, J. Fasel, Families of $\mathbb{A}^{1}$-contractible affine threefolds, Alg. Geom. 5 (2018), no.1, 1-14
\bibitem[F]{F} J. Fasel, The projective bundle theorem for $I^j$-cohomology, J. K-Theory 11(2) (2013), 413-464
\bibitem[HK\O]{HKO} M. Hoyois, A. Krishna, P. A. {\O}stv{\ae}r, $\mathbb{A}^1$-contractibility of Koras-Russell threefolds, Algebraic Geometry 3 (2016), 407-423
\bibitem[KM]{KM} N. M. Kumar and M. P. Murthy, Algebraic cycles and vector bundles over affine three-folds, Ann. of Math. 116 (1982), no. 3, 579–591
\bibitem[KML]{KML} S. Kaliman and L. Makar-Limanov, On the Russell-Koras contractible threefolds, J. Algebraic Geom. 6 (1997), 2, 247–268
\bibitem[KR]{KR} M. Koras and P. Russell, Contractible threefolds and $\mathbb{C}^{\times}$-actions on $\mathbb{C}^3$, J. Algebraic Geom. 6 (1997), 4, 671–695
\bibitem[M]{M} M. P. Murthy. Cancellation problem for projective modules over certain affine algebras. In Algebra, arithmetic and geometry, Part I, II (Mumbai, 2000), volume 16 of Tata Inst. Fund. Res. Stud. Math., pages 493–507. Tata Inst. Fund. Res., Bombay, 2002
\bibitem[MVW]{MVW} C. Mazza, V. Voevodsky, and C. Weibel, Lecture notes on motivic cohomology, Clay Mathematics Monographs, vol. 2, American Mathematical Society, Providence, RI, 2006
\bibitem[Sr]{Sr} V. Srinivas, Torsion $0$-cycles on affine varieties in characteristic $p$, J. Algebra 120 (1989), 428-432
\bibitem[Sy]{Sy} T. Syed, Motivic cohomology of cyclic coverings, Advances in Mathematics 487 (2026), 110767
\bibitem[T]{T} B. Totaro, Non-injectivity of the map from the Witt group of a variety to the Witt group of its function field, J. Inst. Math. Jussieu 2(3) (2003), 483-493
\bibitem[Z]{Z} M. Zaidenberg, Exotic algebraic structures on affine spaces, Algebra i Analiz 11 (1999), 3-73
\end{thebibliography}
\end{document}